\documentclass[11pt]{article}

\usepackage[english]{babel}
\usepackage[utf8]{inputenc}
\usepackage[T1]{fontenc}
\usepackage{amsmath, amssymb, amsthm}
\usepackage{geometry}
\usepackage{mathrsfs}
\usepackage{hyperref}
\usepackage{xcolor}
\usepackage{graphicx}
\newtheorem{theorem}{Theorem}
\newtheorem{lemma}{Lemma}
\newtheorem{proposition}{Proposition}
\newtheorem{remark}{Remark}

\newcommand{\dd}{\mathrm{d}}

\title{Existence of Traveling Waves for a Diffusive Nicholson Equation with Harvesting}
 
\author{Adri\'an G\'omez\thanks{e-mail: \texttt{agomez@ubiobio.cl}}
\quad Hel\'i Elorreaga\thanks{e-mail: \texttt{helorreaga@ubiobio.cl}}
\quad C\'esar Guayasam\'in\thanks{e-mail: \texttt{caguayasamin@outlook.com}}\\[1mm]
Departamento de Matem\'atica, Universidad del B\'io-B\'io, Casilla 5-C, Concepci\'on, Chile}

\begin{document}

\maketitle
\begin{abstract}
We prove the existence of traveling wave solutions for a diffusive
Nicholson model with a delayed linear harvesting term and two
delayed-response channels. Our approach combines monotone heteroclinic
connections of the associated scalar delay equation with the large-speed
traveling-wave theorem of Faria, Huang and Wu. For
\(0<\sigma<r\), we use a previous upper- and lower-solution construction
under an explicit parameter condition. For \(0<r\leq\sigma\), we
establish a new monotone heteroclinic connection.

To fit the abstract framework, we embed the scalar equation into an
auxiliary two-dimensional reaction--diffusion system and prove that every
relevant traveling wave lies on the invariant diagonal. We then determine
an explicit parameter region in which the required spectral hypotheses
hold. Consequently, the original diffusive Nicholson equation admits
traveling waves connecting the trivial and positive equilibria for all
sufficiently large wave speeds.
\end{abstract}


\section{Introduction}

Delay differential equations arise naturally in population dynamics, where the
present growth rate of a species may depend on past population densities due to
maturation times, delayed recruitment, resource regeneration, or delayed
intervention mechanisms. A paradigmatic example is Nicholson's blowflies
equation, introduced to model the dynamics of adult blowfly populations and later formulated in the now classical form
\[
    x'(t)
    =
    -\delta x(t)
    +
    \rho x(t-r)e^{- x(t-r)}.
\]
Here \(\delta>0\) represents the adult mortality rate, \(\rho>0\) is
the maximum per-capita recruitment rate, and \(r>0\) is the maturation delay. Since the
work of Gurney, Blythe and Nisbet \cite{GurneyBlytheNisbet1980}, Nicholson-type
models have become a central class of delayed population equations.

The corresponding diffusive Nicholson equation has also attracted considerable
attention, especially in connection with the existence of traveling wave fronts.
In its simplest normalized form, it is given by
\[
    N_t(t,x)
    =
    N_{xx}(t,x)
    -
    \delta N(t,x)
    +
    \rho N(t-r,x)e^{-N(t-r,x)}.
\]
Traveling waves are solutions of the form
\[
    N(t,x)=\phi(x+ct),
\]
where \(c>0\) is the wave speed and the profile \(\phi\) connects two steady
states. Such waves describe spatial transitions between different population
regimes and are important in understanding propagation phenomena in delayed
reaction--diffusion models. Various approaches have been developed to prove
their existence, including upper and lower solution methods, monotone
iteration schemes, and nonstandard orderings adapted to delayed and
non-quasi-monotone nonlinearities; see, for instance,
\cite{Schaaf1987,WuZou2001,SoZou2001}.

A different and particularly useful point of view was developed by Faria,
Huang and Wu \cite{FariaHuangWu2006}. Their theory establishes a connection
between traveling waves of delayed reaction--diffusion equations and
heteroclinic solutions of the associated delay differential equation. Roughly
speaking, under suitable spectral hypotheses at the equilibria, the existence of
a heteroclinic orbit for the temporal delay equation implies the existence of
traveling waves for the delayed reaction--diffusion equation for all sufficiently
large wave speeds. This framework is especially convenient when the
heteroclinic connection is already known by independent methods. It was used,
for example, in the scalar setting by Faria and Trofimchuk
\cite{FariaTrofimchuk2007}, who studied positive heteroclinic solutions and
traveling waves for population models with a single delay, including
Nicholson-type equations in parameter regimes where the birth function is not monotone.

In the present work we consider a diffusive Nicholson equation with a delayed
linear harvesting term,
\begin{equation}\label{eq:nicholson_diffusive}
    n_t(t,x)
    =
    n_{xx}(t,x)
    -\delta n(t,x)
    -H n(t-\sigma,x)
    +\rho n(t-r,x)e^{-n(t-r,x)},
\end{equation}
where the parameter \(H>0\) represents the harvesting intensity, while \(\sigma>0\) is the
delay associated with the harvesting response. The associated scalar delay
differential equation is
\begin{equation}\label{eq:scalar_dde}
    n'(t)
    =
    -\delta n(t)
    -H n(t-\sigma)
    +\rho n(t-r)e^{-n(t-r)}.
\end{equation}
If the inequality  $\frac{\rho}{\delta+H}>1$ holds, then \eqref{eq:scalar_dde} has two nonnegative equilibria,
\begin{equation}\label{equilibria}
    0
    \qquad\text{and}\qquad
    \kappa=
    \ln\left(\frac{\rho}{\delta+H}\right).
\end{equation}

The existence of monotone heteroclinic connections for
\eqref{eq:scalar_dde} was established in
\cite{GomezGuayasamin2025} for the distinct-delay regime \(0<\sigma<r\),
under suitable restrictions on the delays and parameters. That result was proved
by adapting the upper--lower solution and monotone iteration framework of
Wu and Zou \cite{WuZou2001} to a Nicholson equation with delayed harvesting. In the present
paper, we use this scalar heteroclinic connection as the main input for the
Faria--Huang--Wu theory. We also treat the complementary regime
\(0<r\leq\sigma\). This regime is not obtained as a direct limit of the
distinct-delay theorem for \(0<\sigma<r\), but it can be handled by a unified upper--lower solution construction.

The case \(r\leq\sigma\), in which the recruitment delay is shorter than the
harvesting delay, requires a different lower solution from the one used for
\(0<\sigma<r\). Indeed, the latter has the form
\[
\underline\phi(t)=
\begin{cases}
\alpha(1-e^{\varepsilon(t-t_0)})e^{\lambda t}, & t\le t_0,\\[4pt]
0, & t>t_0.
\end{cases}
\]
When \(0<\sigma<r\), the harvesting term \(\underline\phi(t-\sigma)\) vanishes
before the recruitment term \(\underline\phi(t-r)\), which is consistent with
the lower-solution inequality. In contrast, if \(r<\sigma\), then for
\(
    t_0+r\leq t< t_0+\sigma
\)
we have
\[
    \underline\phi(t)=0,
    \qquad
    \underline\phi(t-r)=0,
    \qquad
    \underline\phi(t-\sigma)>0.
\]
Hence
\[
\begin{aligned}
&-\underline\phi'(t)-\delta\underline\phi(t)
-H\underline\phi(t-\sigma)
+\rho\underline\phi(t-r)e^{-\underline\phi(t-r)}
\\
&\qquad =
-H\underline\phi(t-\sigma)<0,
\end{aligned}
\]
which contradicts the lower-solution inequality required in Theorem 2.1 of
\cite{GomezGuayasamin2025}. We therefore construct a lower solution with a
positive exponentially decaying right tail. The same construction also covers
the endpoint \(r=\sigma\), and hence yields a unified argument for
\(0<r\leq\sigma\); see Appendix \ref{app:heteroclinic_r_le_sigma}.

A second difficulty appears at the matching point of the lower solution.
The issue is not its lack of monotonicity, since non-monotone lower
solutions are compatible with the classical Wu--Zou approach. Rather,
our construction produces a downward jump in the derivative,
\[
    \underline\varphi'(0^-)
    >
    \underline\varphi'(0^+),
\]
whereas the Wu--Zou \cite{WuZouErratum2001} formulation requires
\[
    \underline\varphi'(0^-)
    \leq
    \underline\varphi'(0^+).
\]
Thus, the present lower solution falls outside that formulation.
However, the version of the monotone iteration method introduced in
\cite{GomezGuayasamin2025} does not impose this derivative junction
condition.

The main purpose of this paper is therefore to prove the existence of
traveling wave solutions for \eqref{eq:nicholson_diffusive} by combining two
ingredients. The first ingredient is the existence of a scalar heteroclinic
solution \(n_*\) of \eqref{eq:scalar_dde} connecting \(0\) to \(\kappa\).
The second ingredient is the abstract theorem of Faria--Huang--Wu, which
converts this heteroclinic connection, together with appropriate spectral
conditions, into traveling wave solutions for the diffusive equation when the
wave speed is sufficiently large.

A technical difficulty is that the delayed harvesting term and the delayed
recruitment term appear with different delays. To fit the Faria--Huang--Wu
framework, we embed the scalar equation into a suitable two-dimensional
auxiliary reaction--diffusion system. The scalar heteroclinic \(n_*\) is then embedded as a diagonal heteroclinic of the associated delay system. After applying the Faria--Huang--Wu theorem to the auxiliary system, we prove that every heteroclinic traveling wave of this auxiliary system lies on the diagonal. Consequently, the traveling wave obtained for the auxiliary system is in fact a traveling wave of the original scalar equation \eqref{eq:nicholson_diffusive}.

Our main result can now be stated as follows.
\begin{theorem}\label{thm:main_traveling_wave}
Assume that $1<\frac{\rho}{\delta+H}\leq e$. Let \(\sigma_0>0\) be the unique number satisfying
\[
    H\sigma_0e^{1+\delta\sigma_0}=1.
\]
Let \(r_*>0\) be the spectral threshold defined in
\eqref{rcrit}. Suppose that
\[
    0<\sigma\leq\sigma_0,
    \qquad
    0<r<r_*.
\]
If \(\sigma<r\), assume in addition that
\[
    H>\rho e^{2\lambda(\sigma-r)},
\]
where \(\lambda>0\) is the unique positive root of
\[
    -\lambda-\delta-He^{-\sigma\lambda}
    +\rho e^{-r\lambda}=0.
\]
Then there exists \(c^*>0\) such that, for every \(c>c^*\), equation
\eqref{eq:nicholson_diffusive} admits a traveling wave solution
$n(t,x)=\phi(x+ct)$ connecting the trivial equilibrium to the positive equilibrium \(\kappa\). More precisely,
\[
    \lim_{s\to-\infty}\phi(s)=0,
    \qquad
    \lim_{s\to+\infty}\phi(s)=\kappa.
\]
\end{theorem}

The remainder of the paper is organized as follows. In Section
\ref{sec:scalar_limiting_equation}, we state the scalar heteroclinic-existence
result used throughout the paper, covering both delay orderings
\(0<\sigma<r\) and \(0<r\leq\sigma\). In Section
\ref{sec:fhw_embedding}, we introduce the auxiliary two-dimensional
reaction--diffusion system, show that its diagonal is invariant, and embed the scalar heteroclinic solution into the associated delay system. Section \ref{sec:traveling_waves_diagonal} derives the traveling-wave profile equations and proves that every wave of the auxiliary system connecting the relevant equilibria lies on the diagonal. In Section
\ref{sec:spectral_hypotheses}, we establish explicit conditions ensuring that the spectral hypotheses \((H1)\) and \((H2)\) of the Faria--Huang--Wu theorem are satisfied. Section \ref{sec:application_fhw} combines these results to
prove Theorem \ref{thm:main_traveling_wave}. Finally, Appendix
\ref{app:heteroclinic_r_le_sigma} provides the detailed upper--lower solution
construction for the complementary regime \(0<r\leq\sigma\).

\section{Heteroclinic connections for non-diffusive Nicholson's equation }
\label{sec:scalar_limiting_equation}
The limiting
delay equation associated with the diffusive Nicholson equation with delayed linear harvesting is given by \eqref{eq:scalar_dde}. The role of this section is to guarantee the existence of a monotone
heteroclinic solution \(n_*\) connecting the trivial equilibrium to the positive
one. As in \cite{GomezGuayasamin2025}, throughout this paper we assume that 
\begin{equation}\label{eq:main_scalar_assumption}
    1<\frac{\rho}{\delta+H}\le e.
\end{equation}
Consequently, $0<\kappa\le1$.

We shall use the following scalar existence result.
\begin{theorem}
\label{prop:scalar_profiles}
Assume that \eqref{eq:main_scalar_assumption} holds and suppose that one of the following alternatives holds.
\begin{description}
    \item[\((D)\)] $0<\sigma<r$ and     
    \begin{enumerate}
        \item $\sigma\in(0,\sigma_0],$ where $\sigma_0$ satisfies the relation $H\sigma_0e^{1+\sigma_0\delta}=1;$
        \item $H> \rho e^{2\lambda(\sigma-r)},$ where \(\lambda>0\) is the positive root of
    \[
        -z-\delta-He^{-\sigma z}+\rho e^{-rz}=0.
    \]
    \end{enumerate}
    \item[\((E)\)] $0<r\leq\sigma\leq\sigma_0$, where \(\sigma_0\) is the
    unique positive number satisfying $H\sigma_0e^{1+\delta\sigma_0}=1$.
\end{description}
Then equation \eqref{eq:scalar_dde} admits a monotone heteroclinic solution
\(n_*:\mathbb R\to\mathbb R\) satisfying
\[
    \lim_{t\to-\infty}n_*(t)=0,
    \qquad
    \lim_{t\to+\infty}n_*(t)=\kappa.
\]
\end{theorem}

\begin{proof}
We consider the two alternatives separately.

Assume first that \((D)\) holds. We use the upper--lower solution
construction developed in the proof of Theorem~1.1 in
\cite{GomezGuayasamin2025}. The only additional point is to verify the
existence of the parameter \(\varepsilon\in(0,\lambda)\) required in
the construction of the lower solution.

Set $d:=r-\sigma>0$. Since \(\lambda>0\) satisfies
\[
    -\lambda-\delta-He^{-\sigma\lambda}
    +\rho e^{-r\lambda}=0,
\]
we have
\[
    \rho e^{-r\lambda}
    =
    \lambda+\delta+He^{-\sigma\lambda}
    >
    He^{-\sigma\lambda},
\]
and hence
\[
    H<\rho e^{-\lambda d}.
\]
On the other hand, the last condition in \((D)\) is equivalent to
\[
    H>\rho e^{-2\lambda d}.
\]
Therefore
\[
    0<
    \frac{1}{d}\log\left(\frac{\rho}{H}\right)-\lambda
    <\lambda.
\]
We may thus choose \(\varepsilon\) such that
\[
    \frac{1}{d}\log\left(\frac{\rho}{H}\right)-\lambda
    <\varepsilon<\lambda.
\]
This choice gives $H>\rho e^{-(\lambda+\varepsilon)d}$, or, equivalently,
\[
    He^{-(\lambda+\varepsilon)\sigma}
    >
    \rho e^{-(\lambda+\varepsilon)r}.
\]
Thus the condition required in the lower-solution construction is
satisfied. The remaining upper--lower solution and compatibility
arguments in \cite{GomezGuayasamin2025} apply without change.
Consequently, Theorem~2.1 therein yields a monotone heteroclinic
solution of \eqref{eq:scalar_dde} connecting \(0\) to \(\kappa\).

Assume now that \((E)\) holds. This case is established in Appendix
\ref{app:heteroclinic_r_le_sigma}; more precisely, it follows from
Theorem \ref{theorem:r_le_sigma}.
\end{proof}

\begin{remark} The preceding argument also revisits the explicit parameter condition \[ H\geq\rho e^{\lambda(\sigma-r)} \] appearing in Theorem~1.1 of \cite{GomezGuayasamin2025}. Indeed, the characteristic identity satisfied by \(\lambda\) gives \[ H<\rho e^{\lambda(\sigma-r)}, \] so the comparison involving \(\lambda\) alone does not provide an admissible parameter range. The lower-solution construction itself requires the existence of some \(0<\varepsilon<\lambda\) such that \[ He^{-(\lambda+\varepsilon)\sigma} \geq \rho e^{-(\lambda+\varepsilon)r}. \] As shown in the proof above, the adjusted condition $ H>\rho e^{2\lambda(\sigma-r)} $ guarantees this requirement. 
With this modification, the lower-solution construction and hence the existence 
argument for \(0<\sigma<r\) in \cite{GomezGuayasamin2025} remain valid. 
\end{remark}

\begin{remark}
The endpoint \(r=\sigma\) is not obtained by simply taking the limit
\(\sigma\to r\) in the distinct-delay theorem. Indeed, the condition $H> \rho e^{2\lambda(\sigma-r)}$ would formally become \(H> \rho\) when \(\sigma=r\), which is incompatible
with \(\rho>\delta+H>H\). It is instead included in the unified construction
for \(0<r\leq\sigma\) given in Appendix
\ref{app:heteroclinic_r_le_sigma}.
\end{remark}

For the rest of the paper, we assume that either \((D)\) or \((E)\) in
Theorem \ref{prop:scalar_profiles} holds. Thus, equation
\eqref{eq:scalar_dde} has a monotone heteroclinic solution
\(n_*:\mathbb R\to\mathbb R\) satisfying
\[
    n_*(-\infty)=0,
    \qquad
    n_*(+\infty)=\kappa.
\]
This scalar orbit is the heteroclinic connection that will be embedded into the associated two-dimensional delay system in the next section.

\section{Embedding into the Faria--Huang--Wu framework}
\label{sec:fhw_embedding}

In this section we  introduce a two-dimensional auxiliary reaction-diffusion system that allows to include  the equation \eqref{eq:nicholson_diffusive} into the Faria-Huang-Wu framework developed in \cite{FariaHuangWu2006}. In order to make the reading self-contained we state partially the  Theorem 1.1 in \cite{FariaHuangWu2006}, according to our context:

\begin{theorem}
\label{thm:fhw}
Consider the delayed reaction--diffusion equation with nonlocal response
\begin{equation}\label{eq:fhw_general_rd}
    u_t(x,t)
    =
    D\Delta u(x,t)
    +
    F\left(
    u(x,t),
    \int_{-r}^{0}\int_{\Omega}
    d\eta(\theta)d\mu(y)\,
    g(u(x+y,t+\theta))
    \right),
\end{equation}
where \(D=\operatorname{diag}(d_1,\ldots,d_n)\), \(d_i>0\), and where
\(F\) and \(g\) are \(C^k\)-smooth, \(k\ge2\). Let the associated delay
differential equation be
\begin{equation}\label{eq:fhw_associated_dde}
    \dot u(t)
    =
    F\left(
    u(t),
    \int_{-r}^{0}
    d\eta(\theta)\,\mu_{\Omega}\,g(u(t+\theta))
    \right),
\end{equation}
where
\[
    \mu_{\Omega}:=\int_{\Omega}d\mu.
\]
Suppose that \eqref{eq:fhw_associated_dde} has two equilibria \(E_1,E_2\)
and that the following hypotheses hold:

\begin{description}
    \item[\((H1)\)]
    All characteristic roots associated with the equilibrium \(E_2\) have
    negative real parts.

    \item[\((H2)\)]
    The equilibrium \(E_1\) is hyperbolic and its unstable manifold is
    \(M\)-dimensional, with \(M\ge1\).

    \item[\((H3)\)]
    Equation \eqref{eq:fhw_associated_dde} has a heteroclinic solution
    \(u^*:\mathbb R\to\mathbb R^n\) connecting \(E_1\) to \(E_2\), that is,
    \[
        \lim_{t\to-\infty}u^*(t)=E_1,
        \qquad
        \lim_{t\to+\infty}u^*(t)=E_2.
    \]

    \item[\((H4)\)]
    The spatial measure has finite first moment:
    \[
        \left\|
        \int_{\Omega}d|\mu|(y)\,|y|
        \right\|_{\mathbb R^{n\times n}}
        <\infty.
    \]
\end{description}

Then there exists \(c^*>0\) such that, for every unit vector
\(\nu\in\mathbb R^m\) and every \(c>c^*\), equation
\eqref{eq:fhw_general_rd} admits a traveling wave solution
\[
    u(x,t)=U(\nu\cdot x+ct)
\]
connecting \(E_1\) to \(E_2\), namely
\[
    \lim_{s\to-\infty}U(s)=E_1,
    \qquad
    \lim_{s\to+\infty}U(s)=E_2.
\]
\end{theorem}

\begin{remark}
The full statement of Theorem 1.1 in \cite{FariaHuangWu2006} also describes
the local \(M\)-dimensional manifold of traveling waves near the heteroclinic
orbit \(u^*\), as well as smooth dependence on the wave speed. In the present
paper we only use the existence part stated above.
\end{remark}

We now embed the scalar diffusive Nicholson equation \eqref{eq:nicholson_diffusive} into a two-dimensional reaction--diffusion system fitting the abstract framework of Faria--Huang--Wu. The reason for introducing an auxiliary system is that the scalar equation contains two different delayed terms, namely \(n(t-r,x)\) and \(n(t-\sigma,x)\). The
Faria--Huang--Wu formulation is naturally written in terms of a vector-valued response variable; therefore, we separate the recruitment and harvesting delays through two components.

We consider the auxiliary system
\begin{equation}\label{eq:auxiliary_diffusive_system}
\begin{cases}
u_{1,t}(t,x)
=
u_{1,xx}(t,x)
-\delta u_1(t,x)
+\rho u_1(t-r,x)e^{-u_1(t-r,x)}
-Hu_2(t-\sigma,x),
\\[2mm]
u_{2,t}(t,x)
=
u_{2,xx}(t,x)
-\delta u_2(t,x)
+\rho u_1(t-r,x)e^{-u_1(t-r,x)}
-Hu_2(t-\sigma,x).
\end{cases}
\end{equation}

Let $\tau=\max\{r,\sigma\}$ and
\[
    U(t,x)=
    \begin{pmatrix}
    u_1(t,x)\\
    u_2(t,x)
    \end{pmatrix}.
\]
Then \eqref{eq:auxiliary_diffusive_system} can be written as
\begin{equation}\label{eq:abstract_rd_form}
    U_t(t,x)
    =
    D U_{xx}(t,x)
    +
    F\left(
    U(t,x),
    \int_{-\tau}^{0}\int_{\Omega}
    \dd\eta(\theta)\dd\mu(y)\,
    g(U(t+\theta,x+y))
    \right),
\end{equation}
where
\[
    D=I_2,
    \qquad
    g(U)=U.
\]

There is no spatial nonlocality in the present model. We therefore take
\(\Omega=\mathbb R\) and choose the matrix-valued spatial measure
\[
    d\mu(y)=I_2\,\delta_0(dy),
\]
where \(\delta_0\) denotes the unit Dirac measure concentrated at
\(y=0\). Consequently,
\[
    \mu_\Omega
    =
    \int_{\Omega}d\mu(y)
    =
    I_2,
\]
and
\[
    \int_{\Omega}d\mu(y)\,
    g\bigl(U(t+\theta,x+y)\bigr)
    =
    g\bigl(U(t+\theta,x)\bigr).
\]
Moreover,
\begin{equation}\label{H4verificacion}
    \left\|
        \int_{\Omega}d|\mu|(y)\,|y|
    \right\|_{\mathbb R^{2\times2}}
    =0.
\end{equation}
Hence hypothesis \((H4)\) of Theorem \ref{thm:fhw} is automatically
satisfied.

Defining the following  matrix-valued Stieltjes measure
\[
    \dd\eta(\theta)
    =\begin{pmatrix}
    \delta_{-r}(\dd\theta) & 0\\
    0 & \delta_{-\sigma}(\dd\theta)
    \end{pmatrix},
\]
we get 
\[
    \int_{-\tau}^{0} \dd\eta(\theta)\,g(U(t+\theta,x))
    = \begin{pmatrix}
    u_1(t-r,x)\\
    u_2(t-\sigma,x)
    \end{pmatrix}.
\]
Finally, for
\[
    U=
    \begin{pmatrix}
    u_1\\
    u_2
    \end{pmatrix},
    \qquad
    V=
    \begin{pmatrix}
    v_1\\
    v_2
    \end{pmatrix},
\]
we set
\[
    F(U,V)
    =
    \begin{pmatrix}
    -\delta u_1+\rho v_1e^{-v_1}-H v_2\\[1mm]
    -\delta u_2+\rho v_1e^{-v_1}-H v_2
    \end{pmatrix}.
\]
With these definitions, \eqref{eq:abstract_rd_form} is precisely
\eqref{eq:auxiliary_diffusive_system}.

We will show that the diagonal subspace
\[
    \Delta=\{(u_1,u_2)\in\mathbb R^2:\ u_1=u_2\}
\]
is invariant under \eqref{eq:auxiliary_diffusive_system}. Moreover, the
restriction of \eqref{eq:auxiliary_diffusive_system} to \(\Delta\) coincides
with the scalar diffusive Nicholson equation \eqref{eq:nicholson_diffusive}.

\begin{lemma}\label{lem:diagonal_invariant}
The diagonal \(\Delta\) is invariant for
\eqref{eq:auxiliary_diffusive_system}. Moreover, if
\[
    u_1(t,x)=u_2(t,x)=n(t,x),
\]
then \(n(t,x)\) satisfies \eqref{eq:nicholson_diffusive}.
\end{lemma}

\begin{proof}
Let
\(
    z(t,x)=u_1(t,x)-u_2(t,x).
\)
Subtracting the second equation in \eqref{eq:auxiliary_diffusive_system} from
the first one gives
\[
    z_t(t,x)=z_{xx}(t,x)-\delta z(t,x).
\]
Thus, if \(z\equiv0\) initially, then \(z\equiv0\) for all later times. Hence the
diagonal is invariant. On the diagonal, both components are equal to some
function \(n(t,x)\), and the first equation in
\eqref{eq:auxiliary_diffusive_system} becomes exactly
\eqref{eq:nicholson_diffusive}.
\end{proof}

The delay differential system associated with
\eqref{eq:auxiliary_diffusive_system} is
\begin{equation}\label{eq:associated_dde_system}
\begin{cases}
u_1'(t)
=
-\delta u_1(t)
+\rho u_1(t-r)e^{-u_1(t-r)}
-Hu_2(t-\sigma),
\\[2mm]
u_2'(t)
=
-\delta u_2(t)
+\rho u_1(t-r)e^{-u_1(t-r)}
-Hu_2(t-\sigma),
\end{cases}
\end{equation}
and its relevant equilibria are
\(
    E_1=(0,0),\, E_2=(\kappa,\kappa),
\)
where $\kappa$ is given as in \eqref{equilibria}.

\begin{lemma}\label{lem:system_heteroclinic}
Let \(n_*(t)\) be the scalar heteroclinic solution given by Theorem
\ref{prop:scalar_profiles}. Then
\[
    U_*(t)
    =
    \begin{pmatrix}
    n_*(t)\\
    n_*(t)
    \end{pmatrix}
\]
is a heteroclinic solution of the system \eqref{eq:associated_dde_system}
connecting \(E_1=(0,0)\) to \(E_2=(\kappa,\kappa)\).
\end{lemma}

\begin{proof}
Since \(n_*(t)\) solves the scalar delay equation \eqref{eq:scalar_dde},
substitution of
\(
    u_1(t)=u_2(t)=n_*(t)
\)
into \eqref{eq:associated_dde_system} shows that both equations reduce to
\eqref{eq:scalar_dde}. Moreover,
\[
    \lim_{t\to-\infty}U_*(t)=
    \begin{pmatrix}
    0\\0
    \end{pmatrix}
    =E_1,
    \qquad
    \lim_{t\to+\infty}U_*(t)=
    \begin{pmatrix}
    \kappa\\ \kappa
    \end{pmatrix}
    =E_2.
\]
Thus \(U_*\) is the required heteroclinic solution.
\end{proof}

\section{Traveling waves and the diagonal reduction}
\label{sec:traveling_waves_diagonal}

In the auxiliary system \eqref{eq:auxiliary_diffusive_system} the diagonal \(u_1=u_2\) reduces exactly to the original scalar equation. In this section we show that  any heteroclinic traveling wave of the
auxiliary system  \eqref{eq:profile_system}  connecting \(E_1=(0,0)\) to \(E_2=(\kappa,\kappa)\) lies on the diagonal. Hence the waves obtained through the auxiliary system correspond to waves of the original scalar model.

We now seek traveling wave solutions of the auxiliary system
\eqref{eq:auxiliary_diffusive_system} of the form $U(t,x)=\Phi(x+ct)$, where
\[
    \Phi(s)=
    \begin{pmatrix}
    \phi_1(s)\\
    \phi_2(s)
    \end{pmatrix},
    \qquad
    s=x+ct.
\]
Then $U_t=c\Phi'(s)$ and $U_{xx}=\Phi''(s)$. Moreover,
\[
    u_1(t-r,x)=\phi_1(s-cr) \quad
    \text{and} \quad
    u_2(t-\sigma,x)=\phi_2(s-c\sigma).
\]
Thus, the traveling wave profile satisfies
\begin{equation}\label{eq:profile_system}
\begin{cases}
c\phi_1'(s)
=
\phi_1''(s)
-\delta \phi_1(s)
+\rho\phi_1(s-cr)e^{-\phi_1(s-cr)}
-H\phi_2(s-c\sigma),
\\[2mm]
c\phi_2'(s)
=
\phi_2''(s)
-\delta \phi_2(s)
+\rho\phi_1(s-cr)e^{-\phi_1(s-cr)}
-H\phi_2(s-c\sigma).
\end{cases}
\end{equation}
A traveling wave connecting \(E_1\) to \(E_2\) satisfies
\[
    \lim_{s\to-\infty}\Phi(s)=E_1=(0,0),
    \qquad
    \lim_{s\to+\infty}\Phi(s)=E_2=(\kappa,\kappa).
\]

Following Faria--Huang--Wu \cite{FariaHuangWu2006}, we introduce the large-speed scaling
\(V(\tau)=\Phi(c\tau) \), \(  \varepsilon=\frac{1}{c^2} \) and then \eqref{eq:profile_system} becomes
\begin{equation}\label{eq:epsilon_profile_system}
\begin{cases}
v_1'(\tau)
=
\varepsilon v_1''(\tau)
-\delta v_1(\tau)
+\rho v_1(\tau-r)e^{-v_1(\tau-r)}
-Hv_2(\tau-\sigma),
\\[2mm]
v_2'(\tau)
=
\varepsilon v_2''(\tau)
-\delta v_2(\tau)
+\rho v_1(\tau-r)e^{-v_1(\tau-r)}
-Hv_2(\tau-\sigma).
\end{cases}
\end{equation}
Formally, as \(c\to+\infty\), one has \(\varepsilon\to0\), and
\eqref{eq:epsilon_profile_system} reduces to the associated delay system
\eqref{eq:associated_dde_system}. Thus the heteroclinic solution \(U_*\) given
by Lemma \ref{lem:system_heteroclinic} is the singular limit around which
large-speed traveling waves are constructed.

\begin{lemma}\label{lem:no_off_diagonal_waves}
Every traveling wave of the auxiliary system \eqref{eq:auxiliary_diffusive_system}
connecting \(E_1=(0,0)\) to \(E_2=(\kappa,\kappa)\) lies on the diagonal. More
precisely, if
\[
    \Phi(-\infty)=E_1,
    \qquad
    \Phi(+\infty)=E_2,
\]
then
\[
    \phi_1(s)=\phi_2(s),
    \qquad s\in\mathbb R.
\]
\end{lemma}

\begin{proof}
Let
\[
    \psi(s)=\phi_1(s)-\phi_2(s).
\]
Subtracting the two equations in \eqref{eq:profile_system}, we obtain
\[
 \psi''(s)-c\psi'(s)-\delta\psi(s)=0.
\]
The associated characteristic equation is $ \mu^2-c\mu-\delta=0$, whose roots are
\[ \mu_{\pm} = \frac{c\pm\sqrt{c^2+4\delta}}{2}.
\]
Since \(\delta>0\), we have that    
\(
    \mu_-<0<\mu_+.
\)
Let
\[
    \psi(s)=A e^{\mu_+s}+B e^{\mu_-s}.
\]
The connection conditions give
\[
    \psi(-\infty)=0,
    \qquad
    \psi(+\infty)=0.
\]
The term \(A e^{\mu_+s}\) is bounded as \(s\to-\infty\), but it diverges as
\(s\to+\infty\) unless \(A=0\). Similarly, the term \(B e^{\mu_-s}\) is bounded
as \(s\to+\infty\), but diverges as \(s\to-\infty\) unless \(B=0\). Hence
\(A=B=0\), and therefore \(\psi\equiv0\). Thus \(\phi_1=\phi_2\).
\end{proof}

Consequently, if
\[
    \Phi(s)=
    \begin{pmatrix}
    \phi(s)\\
    \phi(s)
    \end{pmatrix},
\]
then \eqref{eq:profile_system} reduces to the scalar profile equation
\begin{equation*}\label{eq:scalar_profile}
    c\phi'(s)
    =
    \phi''(s)
    -\delta\phi(s)
    -H\phi(s-c\sigma)
    +\rho\phi(s-cr)e^{-\phi(s-cr)}.
\end{equation*}
Thus every heteroclinic traveling wave of the auxiliary system gives a traveling wave of the original scalar diffusive Nicholson equation with harvesting \eqref{eq:nicholson_diffusive}.

\section{Spectral analysis}
\label{sec:spectral_hypotheses}
To apply Theorem \ref{thm:fhw}, it remains to verify its spectral
hypotheses at the equilibria $E_1=(0,0)$ and $E_2=(\kappa,\kappa)$
of the associated delay system \eqref{eq:associated_dde_system}.
More precisely, hypothesis \((H1)\) requires the asymptotic stability of
\(E_2\), whereas hypothesis \((H2)\) requires \(E_1\) to be hyperbolic
with a nontrivial finite-dimensional unstable manifold.

The characteristic equations associated with the
linearized part of
\eqref{eq:associated_dde_system} around $E_1$ and $E_2$, respectively are:
\begin{equation*}
    (z+\delta) \chi_1(z)=0, \quad \text{and} \quad (z+\delta) \chi_2(z)=0,
\end{equation*}
where 
\[\begin{aligned}
    \chi_1(z)&=z+\delta+H e^{-\sigma z}-\rho e^{-rz},\\
    \chi_2(z)&=z+\delta+H e^{-\sigma z}
    -(\delta+H)(1-\kappa)e^{-r z}.
\end{aligned} \]

The factor \(z+\delta\) contributes only the strictly stable
characteristic root \(z=-\delta\). Therefore, the verification of the
hypotheses $(H1)$ and ($H_2$) reduces to two scalar spectral problems:
we must show that all the roots of \(\chi_2\) lie in the open left
half-plane, and that \(\chi_1\) has no roots on the imaginary axis and
exactly one root, counted with multiplicity, in the open right
half-plane.

The following lemma records the spectral conclusion needed for the
application of Theorem \ref{thm:fhw}. Its proof will follow from the
results established in the next two subsections.

\begin{lemma}
\label{lem:fhw_spectral_conditions}
Assume that $1<\frac{\rho}{\delta+H}\leq e$, and let \(\sigma_0>0\) be the unique number satisfying
\begin{equation}
\label{eq:sigma0_spectral}
H\sigma_0e^{1+\delta\sigma_0}=1.
\end{equation}
Suppose that
$$
    0\leq\sigma\leq\sigma_0,
    \qquad
    0\leq r<r_*,
$$
where \(r_*\) is defined in \eqref{rcrit}. Then hypotheses \((H1)\)
and \((H2)\) of Theorem \ref{thm:fhw} are satisfied for the associated
delay system \eqref{eq:associated_dde_system}. More precisely, all
characteristic roots associated with \(E_2=(\kappa,\kappa)\) have
negative real parts, whereas \(E_1=(0,0)\) is hyperbolic and its
unstable manifold is one-dimensional. In particular, \(M=1\).
\end{lemma}

\subsection{Stability of the positive equilibrium}

We begin with the positive equilibrium \(E_2=(\kappa,\kappa)\).
At this equilibrium, the derivative of the Nicholson recruitment term is
$$
   q:= b'(\kappa)
    =
    \rho e^{-\kappa}(1-\kappa)
    =
    (\delta+H)(1-\kappa).
$$
Since \(0<\kappa\leq1\), this coefficient is nonnegative and strictly
smaller than \(\delta+H\). This strict inequality is the key observation
for controlling the characteristic roots at \(E_2\).

To verify hypothesis \((H1)\), we must prove that every root of
\(\chi_2\) has negative real part. The following proposition shows that
the upper bound \(\sigma\leq\sigma_0\), which already appears in the
construction of the scalar heteroclinic connection, is sufficient for
this purpose. 
\begin{proposition}
\label{prop:sigma0_implies_spectral_stability}
Let \(\sigma_0>0\) be the unique number satisfying \eqref{eq:sigma0_spectral}. Then, for every \(r\geq0\) and every $0\leq\sigma\leq\sigma_0$,
all the roots of the characteristic equation
\begin{equation}
\label{eq:char_positive_sigma0}
    \chi_2(z)
    =
    z+\delta
    +H e^{-\sigma z}
    -q e^{-r z}
    =0,
\end{equation}
have negative real parts.
\end{proposition}

\begin{proof}  We first show that no root can lie on the imaginary axis for \(0<\sigma\leq\sigma_0\). Notice that 
\[ \chi_2(0) = \delta+H-q = (\delta+H)\kappa >0, \] 
so \(z=0\) is not a root. Suppose that \(z=i\omega\), with \(\omega>0\), is a root. Then \[ \delta+i\omega +He^{-i\omega\sigma} = qe^{-i\omega r}, \] and hence 
\[ \begin{aligned} q^2-(\delta+H)^2 &= \left| \delta+i\omega+He^{-i\omega\sigma} \right|^2 -(\delta+H)^2 \\ &= \omega^2 -2H\delta\bigl(1-\cos(\omega\sigma)\bigr) -2H\omega\sin(\omega\sigma). 
\end{aligned} \]
Using 
\[ 1-\cos x\leq\frac{x^2}{2}, \qquad \sin x\leq x, \qquad x\geq0, \] 
we obtain 
\[ \begin{aligned} 
q^2-(\delta+H)^2 &\geq \omega^2 \left[ 1-H\sigma(\delta\sigma+2) \right]. 
\end{aligned} \]
Since \(s\mapsto Hs e^{1+\delta s}\) is increasing and \(\sigma\leq\sigma_0\), \[ H\sigma e^{1+\delta\sigma}\leq1. \] Moreover, \[ e^{1+x}>x+2, \qquad x\geq0, \] and therefore, for \(\sigma>0\), \[ H\sigma(\delta\sigma+2) < H\sigma e^{1+\delta\sigma} \leq1. \] Thus \[ q^2-(\delta+H)^2>0, \] which contradicts \(q<\delta+H\). Hence \(\chi_2\) has no purely imaginary roots for \(0<\sigma\leq\sigma_0\). It remains to identify the stable side. When \(\sigma=0\), 
\[ \chi_2(z) =z+\delta+H-qe^{-r z}. \] 
If \(z=x+iy\) were a root with \(x\geq0\), then 
\[ |z+\delta+H| = qe^{-rx} \leq q, \] whereas 
\[ |z+\delta+H| \geq \delta+H > q, \] 
a contradiction. Thus all characteristic roots have negative real parts when \(\sigma=0\). For fixed \(r\geq0\), the number of characteristic roots in the open right half-plane can change as \(\sigma\) varies only when a root crosses the imaginary axis (see Corollary 2.3 in \cite{RuanWei2003}). Since no such crossing occurs for \(0\leq\sigma\leq\sigma_0\), all characteristic roots remain in the open left half-plane throughout this interval. \end{proof}

\begin{remark}
\label{rem:sigma0_spectral_quasimonotonicity}
The relevance of Proposition
\ref{prop:sigma0_implies_spectral_stability} is that the same upper bound on
the harvesting delay, $0 < \sigma\leq\sigma_0$,
plays two different roles in the analysis. On the one hand, it guarantees
the existence of \(\mu>0\) such that
\[
    \mu-\delta-He^{\mu\sigma}\geq0,
\]
which is the condition used in the exponential quasi-monotonicity argument.
On the other hand, it guarantees that all the roots of equation \eqref{eq:char_positive_sigma0} have a negative real part. Thus, within the parameter region already required for
the construction of the monotone heteroclinic connection, no additional
restriction on \(\sigma\) is needed in order to verify the spectral
stability assumption at \(E_2\).
\end{remark}

\subsection{Hyperbolicity of the zero equilibrium}
\label{subsec:hyperbolicity_zero}

We now turn to the zero equilibrium \(E_1=(0,0)\). It is straightforward to show that, under the condition $\rho>\delta+H$, the equation $\chi_1(z)=0$ has a unique real and positive root $\lambda$. Moreover, under the standing condition \(0\leq\sigma\leq\sigma_0\), this positive
zero is unique. In fact,
$$
    H\sigma
    <
    H\sigma e^{1+\delta\sigma}
    \leq1.
$$
Consequently, for every \(x\geq0\),
$$
\begin{aligned}
    \chi_1'(x)
    &=
    1-H\sigma e^{-\sigma x}
    +\rho r e^{-rx}
    \\
    &\geq
    1-H\sigma
    >0.
\end{aligned}
$$
Hence \(\chi_1\) is strictly increasing on \([0,+\infty)\) and has a
unique positive real root, which we denote by \(\lambda\).

The existence of this root shows that the unstable manifold of \(E_1\)
has dimension at least one. To verify hypothesis \((H2)\), however, we
must still exclude roots on the imaginary axis and show that no
additional roots lie in the open right half-plane. We accomplish this
in two steps. First, we parametrize all possible imaginary roots and
obtain the corresponding critical curves in the \((r,\sigma)\)-plane.
We then identify a connected region containing \((r,\sigma)=(0,0)\)
and separated from all these curves. Constancy of the unstable root
count in that region will imply that \(E_1\) has exactly one unstable
characteristic root.

\subsubsection{Critical curves}
\label{subsec:critical_curves_zero}

Possible changes in the number of unstable roots can occur only at
parameter values for which \(\chi_1\) has a root on the imaginary axis.
We therefore begin by determining all such parameter values. Let
\(z=i\omega\). Since
$$
    \chi_1(0)=\delta+H-\rho<0,
$$
the value \(\omega=0\) is not possible. Moreover, the coefficients are
real, so it is enough to consider \(\omega>0\).

The equation \(\chi_1(i\omega)=0\) can be written as
\[
\delta+i\omega
    =\rho e^{-i\omega r}-H e^{-i\omega\sigma}.
\]
Hence the triangle inequality gives
\[
    \rho-H
    \leq
    \sqrt{\delta^2+\omega^2}
    \leq
    \rho+H.
\]
Therefore purely imaginary roots can occur only for
\[
    \omega\in[\omega_-,\omega_+],
\]
where
\[
    \omega_-=
    \sqrt{(\rho-H)^2-\delta^2},
    \qquad
    \omega_+=
    \sqrt{(\rho+H)^2-\delta^2}.
\]
By \eqref{eq:main_scalar_assumption}, \(\rho>\delta+H\), and hence
\(\omega_->0\).

For \(\omega\in[\omega_-,\omega_+]\), set
\[
    R(\omega):=\sqrt{\delta^2+\omega^2},
    \qquad
    \alpha(\omega)
    :=
    \arctan\left(\frac{\omega}{\delta}\right),
\]
and define
\[
\begin{aligned}
    \beta_\rho(\omega)
    &:=
    \arccos\left(
        \frac{\rho^2-H^2+R(\omega)^2}
             {2\rho R(\omega)}
    \right),\\
    \beta_H(\omega)
    &:=
    \arccos\left(
        \frac{\rho^2-H^2-R(\omega)^2}
             {2H R(\omega)}
    \right).
\end{aligned}
\]
These angles arise from the vector identity
\[
    \rho e^{-i\omega r}
    =
    R(\omega)e^{i\alpha(\omega)}
    +
    H e^{-i\omega\sigma}.
\]
For each admissible frequency there are two possible orientations of
the corresponding triangle, which we denote by
\(\varepsilon\in\{-1,1\}\). Comparing arguments gives
\begin{equation}
\label{eq:critical_curves_zero}
    r(\omega)
    =
    \frac{-\alpha(\omega)
          +\varepsilon\beta_\rho(\omega)+2\pi n}{\omega},
    \qquad
    \sigma(\omega)
    =
    \frac{-\alpha(\omega)
          +\varepsilon\beta_H(\omega)+2\pi m}{\omega},
\end{equation}
where \(m,n\in\mathbb Z\). Only the portions satisfying
\(r,\sigma\geq0\) are admissible.

The following relations will be used repeatedly.

\begin{lemma}
\label{lem:angle_relations_zero}
For every \(\omega\in[\omega_-,\omega_+]\),
\[
    0\leq\beta_\rho(\omega)
    <
    \alpha(\omega)
    <
    \frac{\pi}{2}.
\]
Moreover, if
\[
    \gamma(\omega)
    :=
    \beta_H(\omega)-\beta_\rho(\omega),
\]
then
\[
    \gamma(\omega)
    =
    \arccos\left(
        \frac{\rho^2+H^2-\delta^2-\omega^2}
             {2\rho H}
    \right),
\]
so that
\[
    \gamma(\omega_-)=0,
    \qquad
    \gamma(\omega_+)=\pi,
\]
and \(\gamma\) is strictly increasing on
\((\omega_-,\omega_+)\).
\end{lemma}

\begin{proof}
Since $\frac{\omega}{\delta}>0$, we have $0< \alpha(\omega)=\arctan(\frac{\omega}{\delta}) < \frac{\pi}{2}$. On the other hand, since 
\[
    \cos\alpha(\omega)
    =
    \frac{\delta}{R(\omega)},
\]
we have
\[
    \cos\beta_\rho(\omega)-\cos\alpha(\omega)
    =
    \frac{(\rho-\delta)^2-H^2+\omega^2}
         {2\rho R(\omega)}
    >0,
\]
because \(\rho-\delta>H\). Hence
\(\beta_\rho(\omega)<\alpha(\omega)\).

The angle between the sides of lengths \(\rho\) and \(H\) in the
corresponding triangle is
\(\gamma=\beta_H-\beta_\rho\). The law of cosines gives
\[
    R(\omega)^2
    =
    \rho^2+H^2-2\rho H\cos\gamma(\omega),
\]
which yields the stated expression for \(\gamma\). Since
\[
    R(\omega_-)=\rho-H,
    \qquad
    R(\omega_+)=\rho+H,
\]
we obtain the endpoint values. Finally,
\[
    \gamma'(\omega)
    =
    \frac{\omega}
         {\rho H\sin\gamma(\omega)}
    >0,
    \qquad
    \omega\in(\omega_-,\omega_+).
\]
\end{proof}

Since \(\beta_\rho<\alpha\), every admissible critical curve in
\eqref{eq:critical_curves_zero} satisfies \(n\geq1\).

For later use, let
\[
    N(r,\sigma)
    :=
    \#\left\{
        z:\chi_1(z)=0,\ \Re z>0
    \right\},
\]
where roots are counted with multiplicity. 

Since \(\chi_1(0)=\delta+H-\rho\neq0\), for every \((r,\sigma)\), zero can never be a characteristic root. Moreover, by Corollary 2.3 in \cite{RuanWei2003}, the number \(N(r,\sigma)\) can change only when a root crosses the imaginary axis. Hence \(N\) is constant on every connected region which does not intersect the critical curves.

\subsubsection{A uniform spectral region}
\label{subsec:uniform_spectral_region}

The parametrisation \eqref{eq:critical_curves_zero} yields a simple
condition which is independent of the relative size of the two delays.

\begin{proposition}
\label{prop:uniform_rcrit_zero}
Define
\begin{equation}\label{rcrit}
    r_{*}
    :=
    \min_{\omega\in[\omega_-,\omega_+]}
    \frac{
        2\pi-\alpha(\omega)-\beta_\rho(\omega)
    }{\omega}.
\end{equation}
Then, for every \(\sigma\geq0\),
\[
    0\leq r<r_{*}
\]
implies that \(\chi_1\) has no roots on the imaginary axis and $N(r,\sigma)=1$. Consequently, \(E_1\) is hyperbolic and its unstable dimension is
\(M=1\). 
\end{proposition}

\begin{proof}
Let \((r,\sigma)\) belong to a critical curve. Since \(n\geq1\), if
\(\varepsilon=-1\), then
\[
    r
    =
    \frac{
        2\pi n-\alpha(\omega)-\beta_\rho(\omega)
    }{\omega}
    \geq
    r_{*}.
\]
If \(\varepsilon=1\), then
\[
    r
    =
    \frac{2\pi n-\alpha(\omega)+\beta_\rho(\omega)}{\omega}
    \geq 
    \frac{2\pi-\alpha(\omega)-\beta_\rho(\omega)}{\omega}
    \geq
    r_{*}.
\]
Thus no critical curve intersects the connected strip
\[
    0\leq r<r_{*},
    \qquad
    \sigma\geq0.
\]

At \(r=\sigma=0\),
\[
    \chi_1(z)
    =
    z+\delta+H-\rho,
\]
and hence the only root is
\[
    z_*=\rho-\delta-H>0.
\]
Therefore \(N(0,0)=1\), and by constancy of the unstable root count, $N(r,\sigma)=1$ throughout the strip. This also completes the proof of Lemma
\ref{lem:fhw_spectral_conditions}.
\end{proof}

\begin{remark}
Since $\beta_\rho(\omega)<\alpha(\omega)<\frac{\pi}{2}$,
the critical value $r_{*}$  satisfies
\[
    r_{*}
    >
    \frac{\pi}{\omega_+}.
\]
Hence the explicit condition
\[
    0\leq r\leq
    \frac{\pi}{
        \sqrt{(\rho+H)^2-\delta^2}
    }
\]
is a simpler, although more restrictive, sufficient condition for
\(M=1\), uniformly in \(\sigma\).
\end{remark}

The critical curves and the uniform parameter region obtained above
are illustrated in Figure \ref{fig:spectral_region}.
\begin{figure}[htbp]
    \centering
    \includegraphics[width=0.8\textwidth]{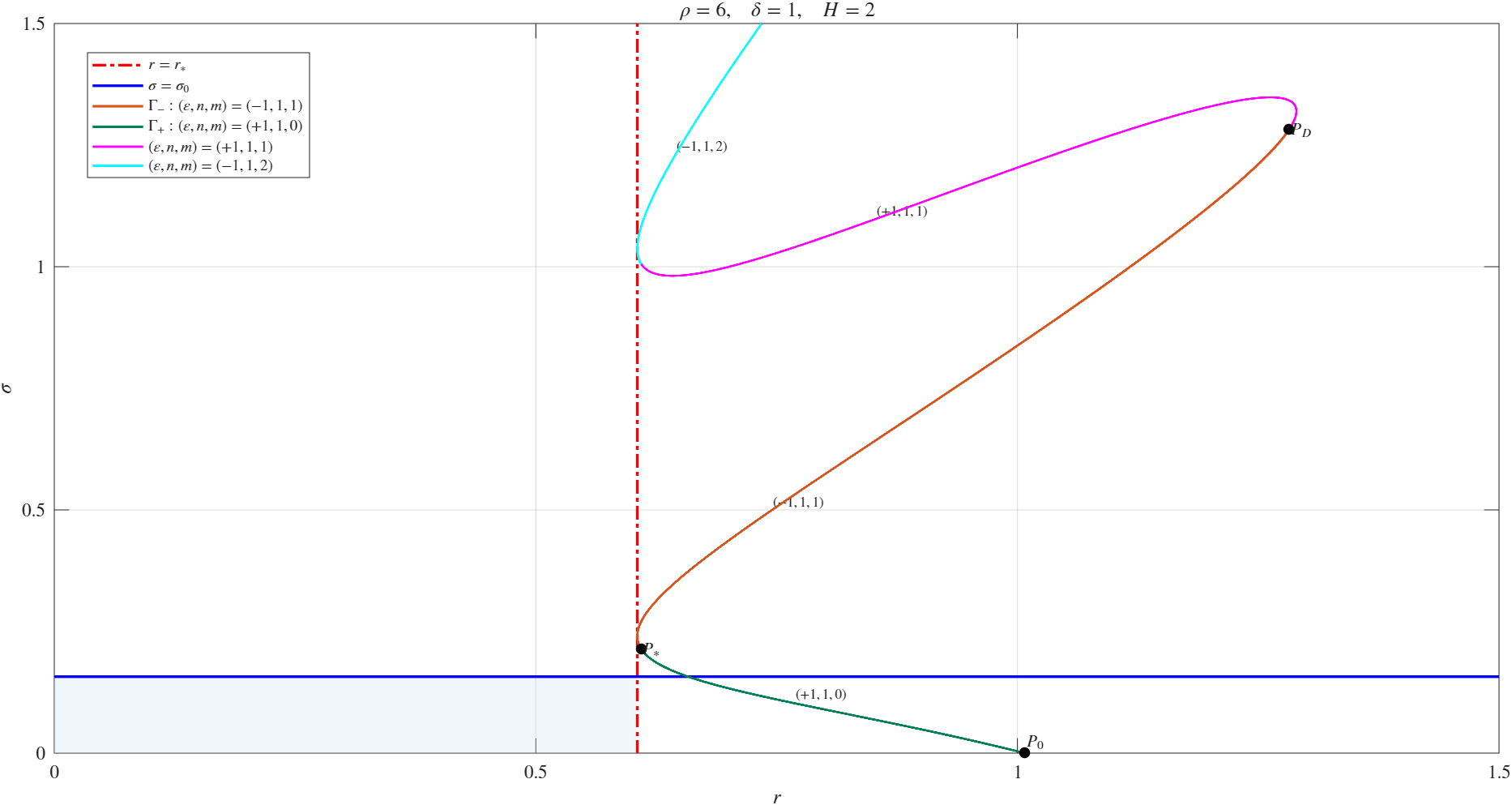}
    \caption{Selected critical curves of
    \(\chi_1(i\omega)=0\) in the \((r,\sigma)\)-plane for
    \(\rho=6\), \(\delta=1\), and \(H=2\). Each solid curve corresponds
    to the indicated triple \((\varepsilon,n,m)\) in the
    parametrisation \eqref{eq:critical_curves_zero}. The dashed vertical
    line represents \(r=r_*\), while the horizontal line represents
    \(\sigma=\sigma_0\). The shaded rectangle $ 0\leq r<r_*$, $
        0\leq\sigma\leq\sigma_0$,
    is the parameter region in which hypotheses \((H1)\) and \((H2)\)
    are simultaneously satisfied.}
    \label{fig:spectral_region}
\end{figure}

\section{Proof of Theorem \ref{thm:main_traveling_wave}}
\label{sec:application_fhw}
\begin{proof}
If \(0<\sigma<r\), the assumptions of the theorem imply alternative
\((D)\) in Theorem \ref{prop:scalar_profiles}. If \(0<r\leq\sigma\),
they imply alternative \((E)\). Therefore, in either case, Theorem
\ref{prop:scalar_profiles} provides a monotone heteroclinic solution
\(n_*\) of \eqref{eq:scalar_dde} connecting \(0\) to \(\kappa\).
Moreover, Lemma \ref{lem:system_heteroclinic} shows that
$$
U_*(t)=
    \begin{pmatrix}
        n_*(t)\\
        n_*(t)
    \end{pmatrix}
$$
is a heteroclinic solution of the associated delay system
\eqref{eq:associated_dde_system} connecting
\(E_1=(0,0)\) to \(E_2=(\kappa,\kappa)\). Hence hypothesis \((H3)\) of
Theorem \ref{thm:fhw} is satisfied.

The nonlinear functions \(F\) and \(g\) defined in Section
\ref{sec:fhw_embedding} are \(C^\infty\), the diffusion matrix is
\(D=I_2\), and the temporal Stieltjes measure is a finite combination of
Dirac masses. Therefore, the auxiliary reaction--diffusion system
\eqref{eq:auxiliary_diffusive_system} has the structural form required in
Theorem \ref{thm:fhw}. Moreover, since the spatial measure is concentrated
at \(y=0\), equation \eqref{H4verificacion} gives

$$
    \left\|
        \int_{\Omega}\dd|\mu|(y)\,|y|
    \right\|=0,
$$

and thus hypothesis \((H4)\) is also satisfied.

By Lemma \ref{lem:fhw_spectral_conditions}, all characteristic roots
associated with \(E_2\) have negative real parts. Therefore hypothesis
\((H1)\) of Theorem \ref{thm:fhw} holds. The same lemma also shows that
\(E_1\) is hyperbolic and that its unstable manifold is one-dimensional.
Hence hypothesis \((H2)\) holds with \(M=1\).

All the hypotheses of Theorem \ref{thm:fhw} are therefore satisfied.
Consequently, there exists \(c^*>0\) such that, for every \(c>c^*\), the
auxiliary system \eqref{eq:auxiliary_diffusive_system} admits a traveling
wave
$$
U(t,x)=\Phi(x+ct),
    \qquad
\Phi(s)=
    \begin{pmatrix}
        \phi_1(s)\\
        \phi_2(s)
    \end{pmatrix},
$$
satisfiying 
\begin{equation}\label{asint}
\lim_{s\to - \infty} \Phi(s)=E_1,\quad \lim_{s\to + \infty} \Phi(s)=E_2.
\end{equation} 

By Lemma \ref{lem:no_off_diagonal_waves}, every traveling wave of the
auxiliary system connecting these two equilibria lies on the diagonal.
It follows that
$$
    \phi_1(s)=\phi_2(s)=:\phi(s),
    \qquad s\in\mathbb R,
$$
and 
$$
    \lim_{s\to-\infty}\phi(s)=0,
    \qquad
    \lim_{s\to+\infty}\phi(s)=\kappa.
$$
Therefore $n(t,x)=\phi(x+ct)$ is a traveling wave solution of the scalar delayed diffusive Nicholson
equation \eqref{eq:nicholson_diffusive}, connecting the trivial equilibrium
to the positive equilibrium \(\kappa\), for every \(c>c^*\).
\end{proof}

\appendix

\section{Heteroclinic connections for $0<r\leq\sigma$}
\label{app:heteroclinic_r_le_sigma}
In this appendix we verify alternative \((E)\) in Theorem
\ref{prop:scalar_profiles}. Throughout the section we assume
\[
1<\frac{\rho}{\delta+H}\leq e,
\qquad
0<r\leq\sigma\leq\sigma_0,
\]
where $\sigma_0$ is the unique solution of $H\sigma_0e^{1+\delta\sigma_0}=1$. We recall that if we set $\kappa=\ln\left(\frac{\rho}{\delta+H}\right)$, then \(0<\kappa\leq1\). 

We apply Theorem 2.1 in
\cite{GomezGuayasamin2025} to equation \eqref{eq:scalar_dde}, written as
\[
-u'(t)+f(u_t)=0,
\]
where
\[
f(\phi)
=
-\delta\phi(0)-H\phi(-\sigma)
+\rho\phi(-r)e^{-\phi(-r)},
\qquad
\phi\in C([-\sigma,0];\mathbb R).
\]

\subsection{Common preliminaries}

\begin{lemma}
\label{lemma:H1_H2_r_le_sigma}
The functional \(f\) satisfies the following condition appearing in  \cite{GomezGuayasamin2025}:
\begin{description}
    \item[\((H1)\)]   $f(\widehat{0})=f(\widehat{\kappa})=0$ and $f(\widehat{u}) \neq 0$ for arbitrary $u \in(0, \kappa)$.
    \item[\((H2)\)]   (Exponential quasi-monotonicity). For $\mu=\frac{1}{\sigma}+\delta > 0$ we have
    $$
    f(\phi)-f(\psi)+\mu[\phi(0)-\psi(0)] \geq 0,
    $$
    for all $\phi, \psi \in C([-\sigma, 0]; \, \mathbb{R})$ such that
    \begin{enumerate}
    \item[(i)] $0 \leq \psi(s) \leq \phi(s) \leq \kappa, \quad  s \in[-\sigma, 0]$,
    \item[(ii)] $e^{\mu s}(\phi(s)-\psi(s))$ is non-decreasing for $s \in[-\sigma, 0]$.
    \end{enumerate}
    \end{description}
\end{lemma}

\begin{proof}
Condition \((H1)\) is satisfied immediately. In order to chek $(H2)$ note that  the function \(s\mapsto Hs e^{1+\delta s}\) is strictly increasing on
\([0,+\infty)\), the number \(\sigma_0\) is well defined and unique. From
\(\sigma\leq\sigma_0\) we obtain
\[
H\sigma e^{1+\delta\sigma}\leq1.
\]
Consequently,
\[
\mu-\delta-He^{\mu\sigma}
=
\frac1\sigma\left(1-H\sigma e^{1+\delta\sigma}\right)
\geq0.
\]

Let \(\phi,\psi\in C([-\sigma,0];\mathbb R)\) satisfy
\[
0\leq\psi(s)\leq\phi(s)\leq\kappa,
\qquad s\in[-\sigma,0],
\]
and suppose that $s\longmapsto e^{\mu s}\bigl(\phi(s)-\psi(s)\bigr)$
is non-decreasing. 

Writing \(D=\phi-\psi\), comparison at \(-\sigma\) and \(0\) gives
\[
D(-\sigma)\leq e^{\mu\sigma}D(0).
\]
The Nicholson function \(h(u)=ue^{-u}\) is non-decreasing on
\([0,\kappa]\), because \(\kappa\leq1\). Therefore
\[
\begin{aligned}
f(\phi)-f(\psi)+\mu D(0)
&=
(\mu-\delta)D(0)-HD(-\sigma)
+\rho\bigl[h(\phi(-r))-h(\psi(-r))\bigr]
\\
&\geq
\left(\mu-\delta-He^{\mu\sigma}\right)D(0)
\geq0.
\end{aligned}
\]
\end{proof}
For consistency with the notation used in
\cite{GomezGuayasamin2025}, we introduce the characteristic function
\[
    \chi(z)
    :=
    -z-\delta-He^{-\sigma z}+\rho e^{-rz}
    =
    -\chi_1(z).
\]
It's not hard to see that $\chi(x)<0$ for every $x>\lambda$.

\subsection{An upper solution}
Let \(\lambda>0\) be the unique zero of $\chi(z)=0$, and let \(\mu\) be as in Lemma
\ref{lemma:H1_H2_r_le_sigma}. Define
\begin{equation}
\label{upper_r_le_sigma}
\overline\varphi(t)
=
\begin{cases}
\displaystyle
\frac{\kappa\mu}{\mu+\lambda}e^{\lambda t},
& t\leq0,
\\[3mm]
\displaystyle
\kappa\left(
1-\frac{\lambda}{\mu+\lambda}e^{-\mu t}
\right),
& t>0.
\end{cases}
\end{equation}

\begin{lemma}
\label{lemma:upper_r_le_sigma}
The function \(\overline\varphi\) is an upper solution of
\eqref{eq:scalar_dde} and belongs to the profile set \(\Gamma\).
\end{lemma}

\begin{proof}
The two branches in \eqref{upper_r_le_sigma}, as well as their first
derivatives, agree at \(t=0\). Hence
\(\overline\varphi\in C^1(\mathbb R)\). Define
\[
\mathcal R[\varphi](t)
:=
-\varphi'(t)-\delta\varphi(t)-H\varphi(t-\sigma)
+\rho\varphi(t-r)e^{-\varphi(t-r)}.
\]
We claim  that \(\mathcal R[\overline\varphi](t)\leq0\).
Indeed, if \(t\leq0\), all delayed arguments lie in the left branch. Since
\(e^{-x}\leq1\) for \(x\geq0\),
\[
\begin{aligned}
\mathcal R[\overline\varphi](t)
&\leq
-\overline\varphi'(t)-\delta\overline\varphi(t)
-H\overline\varphi(t-\sigma)+\rho\overline\varphi(t-r)
\\
&=
\frac{\kappa\mu}{\mu+\lambda}e^{\lambda t}\chi(\lambda)
=0.
\end{aligned}
\]

Suppose now that \(0<t\leq\sigma\). Since
\(0\leq\overline\varphi(s)\leq\kappa\leq1\) for all \(s\), the monotonicity of
\(ue^{-u}\) on \([0,\kappa]\) gives
\[
\rho\overline\varphi(t-r)e^{-\overline\varphi(t-r)}
\leq
\rho\kappa e^{-\kappa}
=
\kappa(\delta+H).
\]
Also, \(t-\sigma\leq0\). It follows that
\[
\mathcal R[\overline\varphi](t)\leq g(t),
\]
where
\[
g(t)
=
H\kappa
-
\frac{\lambda\kappa}{\mu+\lambda}
(\mu-\delta)e^{-\mu t}
-
\frac{H\mu\kappa}{\mu+\lambda}e^{\lambda(t-\sigma)}.
\]
At \(t=\sigma\),
\[
g(\sigma)
=
-\frac{\lambda\kappa}{\mu+\lambda}e^{-\mu\sigma}
\left(\mu-\delta-He^{\mu\sigma}\right)
\leq0.
\]
Furthermore,
\[
\begin{aligned}
g'(t)
&=
\frac{\lambda\mu\kappa}{\mu+\lambda}
(\mu-\delta)e^{-\mu t}
-
\frac{H\lambda\mu\kappa}{\mu+\lambda}
e^{\lambda(t-\sigma)}
\\
&\geq
\frac{H\lambda\mu\kappa}{\mu+\lambda}
\left[
e^{-\mu(t-\sigma)}-e^{\lambda(t-\sigma)}
\right]
\geq0.
\end{aligned}
\]
Thus \(g(t)\leq g(\sigma)\leq0\) on this interval.

Finally, if \(t\geq\sigma\), then \(t-\sigma\geq0\) and, because
\(r\leq\sigma\), also \(t-r\geq0\). Using the same bound for the Nicholson
term, we obtain
\[
\begin{aligned}
\mathcal R[\overline\varphi](t)
&\leq
-\frac{\kappa\lambda\mu}{\mu+\lambda}e^{-\mu t}
-\delta\kappa
\left(1-\frac{\lambda}{\mu+\lambda}e^{-\mu t}\right)
\\
&\quad
-H\kappa
\left(1-\frac{\lambda}{\mu+\lambda}
e^{-\mu(t-\sigma)}\right)
+\kappa(\delta+H)
\\
&=
-\frac{\kappa\lambda}{\mu+\lambda}e^{-\mu t}
\left(\mu-\delta-He^{\mu\sigma}\right)
\leq0.
\end{aligned}
\]
Therefore \(\overline\varphi\) is an upper solution.

It remains to verify that \(\overline\varphi\in\Gamma\). Its limits at
\(\pm\infty\) are \(0\) and \(\kappa\), respectively, and
\(\overline\varphi'(t)>0\) for all \(t\). For \(s>0\), set
\[
\gamma_s(t)
=
e^{\mu t}
\left[\overline\varphi(t+s)-\overline\varphi(t)\right].
\]
If \(t+s\leq0\), then
\[
\gamma_s'(t)
=
\kappa\mu e^{(\mu+\lambda)t}
\left(e^{\lambda s}-1\right)>0.
\]
If \(t\leq0<t+s\), direct differentiation gives
\[
\gamma_s'(t)
=
\kappa\mu e^{\mu t}\left(1-e^{\lambda t}\right)\geq0.
\]
If \(t\geq0\), then
\[
\gamma_s(t)
=
\frac{\kappa\lambda}{\mu+\lambda}
\left(1-e^{-\mu s}\right),
\]
which is independent of \(t\). Since \(\gamma_s\) is continuous at the
transition points, it is non-decreasing on \(\mathbb R\). Hence
\(\overline\varphi\in\Gamma\).
\end{proof}

\subsection{A compatible lower solution}

Choose $\lambda<\nu<2\lambda$. Then \(\chi(\nu)<0\). Also, since \(\chi(0)>0\), continuity allows us to choose
\(\eta<0\), sufficiently close to zero, such that \(\chi(\eta)>0\). Set
\[
C_\eta:=\rho e^{-\eta r}-(\eta+\delta)
       =He^{-\eta\sigma}+\chi(\eta)>0
\]
and
\[
d_1:=e^{\eta r}(\lambda-\eta),
\qquad
d_2:=\chi(\eta)e^{\eta\sigma}.
\]
Both constants are positive. Define
\[
\begin{aligned}
D_1(t)
&:=
(\eta+\delta)e^{\eta t}
-\bigl(\nu+\delta+\chi(\nu)\bigr)e^{\nu t},
&&0\leq t\leq r,
\\
D_2(t)
&:=
He^{\nu(t-\sigma)}-C_\eta e^{\eta t},
&&r\leq t\leq\sigma,
\end{aligned}
\]
and let
\[
M_1:=\max_{0\leq t\leq r}|D_1(t)|,
\qquad
M_2:=\max_{r\leq t\leq\sigma}|D_2(t)|.
\]
When \(r=\sigma\), the second maximum is simply taken over the singleton
\(\{r\}\). Choose \(\theta>0\) so small that
\[
\theta\leq
\min\left\{
\frac12,\,
\frac{\lambda}{2\nu},\,
\frac{d_1}{2\max\{1,M_1\}},\,
\frac{d_2}{2\max\{1,M_2\}}
\right\}.
\]
In particular,
\[
0<\theta<1,
\qquad
\theta\nu\leq\frac{\lambda}{2},
\qquad
\theta M_i\leq\frac{d_i}{2},
\quad i=1,2.
\]

For \(A>0\), define
\begin{equation}
\label{lower_r_le_sigma}
\underline\varphi(t)
=
\begin{cases}
\displaystyle
A\left(e^{\lambda t}-\theta e^{\nu t}\right),
& t\leq0,
\\[3mm]
\displaystyle
A(1-\theta)e^{\eta t},
& t>0.
\end{cases}
\end{equation}

\begin{lemma}
\label{lemma:lower_compatible_r_le_sigma}
There exists \(A_*>0\) such that, for every \(0<A\leq A_*\),
\(\underline\varphi\) is a lower solution of \eqref{eq:scalar_dde} and the
pair \((\underline\varphi,\overline\varphi)\) satisfies the compatible conditions 
\((C1)\)--\((C3)\) in \cite{GomezGuayasamin2025}.
\end{lemma}

\begin{proof}
The two branches in \eqref{lower_r_le_sigma} agree at \(t=0\), so
\(\underline\varphi\) is continuous. For \(t\leq0\),
\[
\underline\varphi(t)
=
Ae^{\lambda t}\left(1-\theta e^{(\nu-\lambda)t}\right)>0,
\]
and, because \(\theta\nu\leq\lambda/2\),
\[
\underline\varphi'(t)
=
Ae^{\lambda t}
\left(\lambda-\theta\nu e^{(\nu-\lambda)t}\right)
\geq
\frac{A\lambda}{2}e^{\lambda t}>0,
\qquad t<0.
\]
For \(t>0\),
\[
\underline\varphi(t)=A(1-\theta)e^{\eta t}>0,
\qquad
\underline\varphi'(t)=A(1-\theta)\eta e^{\eta t}<0.
\]
Thus \(\underline\varphi\) is piecewise \(C^1\), locally absolutely
continuous, and has an essentially bounded derivative.

We now prove that
\(\mathcal R[\underline\varphi](t)\geq0\) for almost every \(t\). We use
repeatedly
\[
xe^{-x}\geq x-x^2,
\qquad x\geq0.
\]

If \(t<0\), all terms are evaluated on the left branch. Since the linear
part acts on \(e^{zt}\) as multiplication by \(\chi(z)\), we have
\[
\begin{aligned}
\mathcal R[\underline\varphi](t)
&\geq
-A\theta\chi(\nu)e^{\nu t}
-\rho\underline\varphi(t-r)^2
\\
&\geq
Ae^{\nu t}
\left[
-\theta\chi(\nu)
-\rho A e^{(2\lambda-\nu)t-2\lambda r}
\right]
\\
&\geq
Ae^{\nu t}
\left[
-\theta\chi(\nu)-\rho A e^{-2\lambda r}
\right].
\end{aligned}
\]
Hence the residual is non-negative provided
\[
A\leq
A_1:=
\frac{\theta[-\chi(\nu)]e^{2\lambda r}}{\rho}.
\]

Suppose that \(0<t<r\). Then the current value lies in the right branch,
whereas both delayed values lie in the left branch. Using
\(\underline\varphi(t-r)\leq A\), we obtain
\[
\mathcal R[\underline\varphi](t)
\geq
AG_1(t)-\rho A^2,
\]
where
\[
\begin{aligned}
G_1(t)
&=
-(\eta+\delta)(1-\theta)e^{\eta t}
-H\left(e^{\lambda(t-\sigma)}-\theta e^{\nu(t-\sigma)}\right)
\\
&\quad
+\rho\left(e^{\lambda(t-r)}-\theta e^{\nu(t-r)}\right)
\\
&=
e^{\eta t}
\left[
(\lambda+\delta)e^{(\lambda-\eta)t}-(\eta+\delta)
\right]
+\theta D_1(t).
\end{aligned}
\]
Since \(t\geq0\),
\[
e^{\eta t}
\left[
(\lambda+\delta)e^{(\lambda-\eta)t}-(\eta+\delta)
\right]
\geq
e^{\eta t}(\lambda-\eta)
\geq d_1.
\]
Therefore
\[
G_1(t)\geq d_1-\theta M_1\geq\frac{d_1}{2},
\]
and the residual is non-negative if
\[
A\leq A_2:=\frac{d_1}{2\rho}.
\]

Consider next \(r\leq t<\sigma\). This interval is empty when
\(r=\sigma\). Here \(t-r\geq0\) and \(t-\sigma<0\). Since
\[
0\leq\underline\varphi(t-r)
=
A(1-\theta)e^{\eta(t-r)}
\leq A,
\]
we obtain
\[
\mathcal R[\underline\varphi](t)
\geq
AG_2(t)-\rho A^2,
\]
where
\[
G_2(t)
=
(1-\theta)C_\eta e^{\eta t}
-He^{\lambda(t-\sigma)}
+\theta He^{\nu(t-\sigma)}.
\]
Writing
\[
G_2(t)
=
C_\eta e^{\eta t}
-He^{\lambda(t-\sigma)}
+\theta D_2(t),
\]
and using
\(C_\eta=He^{-\eta\sigma}+\chi(\eta)\), we get
\[
\begin{aligned}
C_\eta e^{\eta t}
-He^{\lambda(t-\sigma)}
&=
H\left[
e^{\eta(t-\sigma)}-e^{\lambda(t-\sigma)}
\right]
+\chi(\eta)e^{\eta t}
\\
&\geq
\chi(\eta)e^{\eta\sigma}
=d_2.
\end{aligned}
\]
Consequently,
\[
G_2(t)\geq d_2-\theta M_2\geq\frac{d_2}{2},
\]
so the residual is non-negative if
\[
A\leq A_3:=\frac{d_2}{2\rho}.
\]

Finally, let \(t\geq\sigma\). Since \(r\leq\sigma\), all terms are evaluated
on the right branch; at a zero delayed argument the two branches have the same
value. Hence
\[
\begin{aligned}
\mathcal R[\underline\varphi](t)
&\geq
A(1-\theta)\chi(\eta)e^{\eta t}
-\rho A^2(1-\theta)^2e^{2\eta(t-r)}
\\
&=
A(1-\theta)e^{\eta t}
\left[
\chi(\eta)
-\rho A(1-\theta)e^{\eta(t-2r)}
\right].
\end{aligned}
\]
Because \(\eta<0\) and \(t\geq\sigma\), this is non-negative provided
\[
A\leq
A_4:=
\frac{\chi(\eta)e^{-\eta(\sigma-2r)}}{\rho(1-\theta)}.
\]
Thus \(\underline\varphi\) is a lower solution whenever
\[
0<A\leq A_{\mathrm{low}}
:=
\min\{A_1,A_2,A_3,A_4\}.
\]

It remains to verify compatibility. Set
\[
c:=\frac{\kappa\mu}{\mu+\lambda}
\qquad\text{and}\qquad
A_*:=\min\{A_{\mathrm{low}},c\}.
\]
For \(0<A\leq A_*\) and \(t\leq0\),
\[
0\leq
\underline\varphi(t)
\leq
Ae^{\lambda t}
\leq
ce^{\lambda t}
=
\overline\varphi(t).
\]
For \(t\geq0\),
\[
0<
\underline\varphi(t)
\leq A
\leq c
=
\overline\varphi(0)
\leq\overline\varphi(t).
\]
Since \(\overline\varphi(t)\leq\kappa\), condition \((C1)\) follows.
Furthermore,
\(\underline\varphi(0)=A(1-\theta)>0\), which proves \((C2)\).

For \((C3)\), define
\[
F(t)
:=
e^{\mu t}
\left[
\overline\varphi(t)-\underline\varphi(t)
\right].
\]
If \(t\leq0\), then
\[
F(t)
=
(c-A)e^{(\mu+\lambda)t}
+
A\theta e^{(\mu+\nu)t},
\]
and therefore \(F'(t)\geq0\). If \(t\geq0\), then
\[
F'(t)
=
e^{\mu t}
\left[
\kappa\mu
-A(1-\theta)(\mu+\eta)e^{\eta t}
\right].
\]
This is positive when \(\mu+\eta\leq0\). If \(\mu+\eta>0\), then
\(e^{\eta t}\leq1\) and
\[
A(1-\theta)(\mu+\eta)
<
\frac{\kappa\mu}{\mu+\lambda}(\mu+\eta)
<
\kappa\mu,
\]
because \(\eta<0<\lambda\). Hence \(F'(t)>0\) also in this case. Since
\(F\) is continuous at \(t=0\), it is non-decreasing on \(\mathbb R\), and
\((C3)\) follows.
\end{proof}

A illustration of the lower solution and its residual is
shown in Figure \ref{fig:lower_solution_verification}.

\begin{figure}[htbp]
    \centering
    \includegraphics[width=0.8\textwidth]{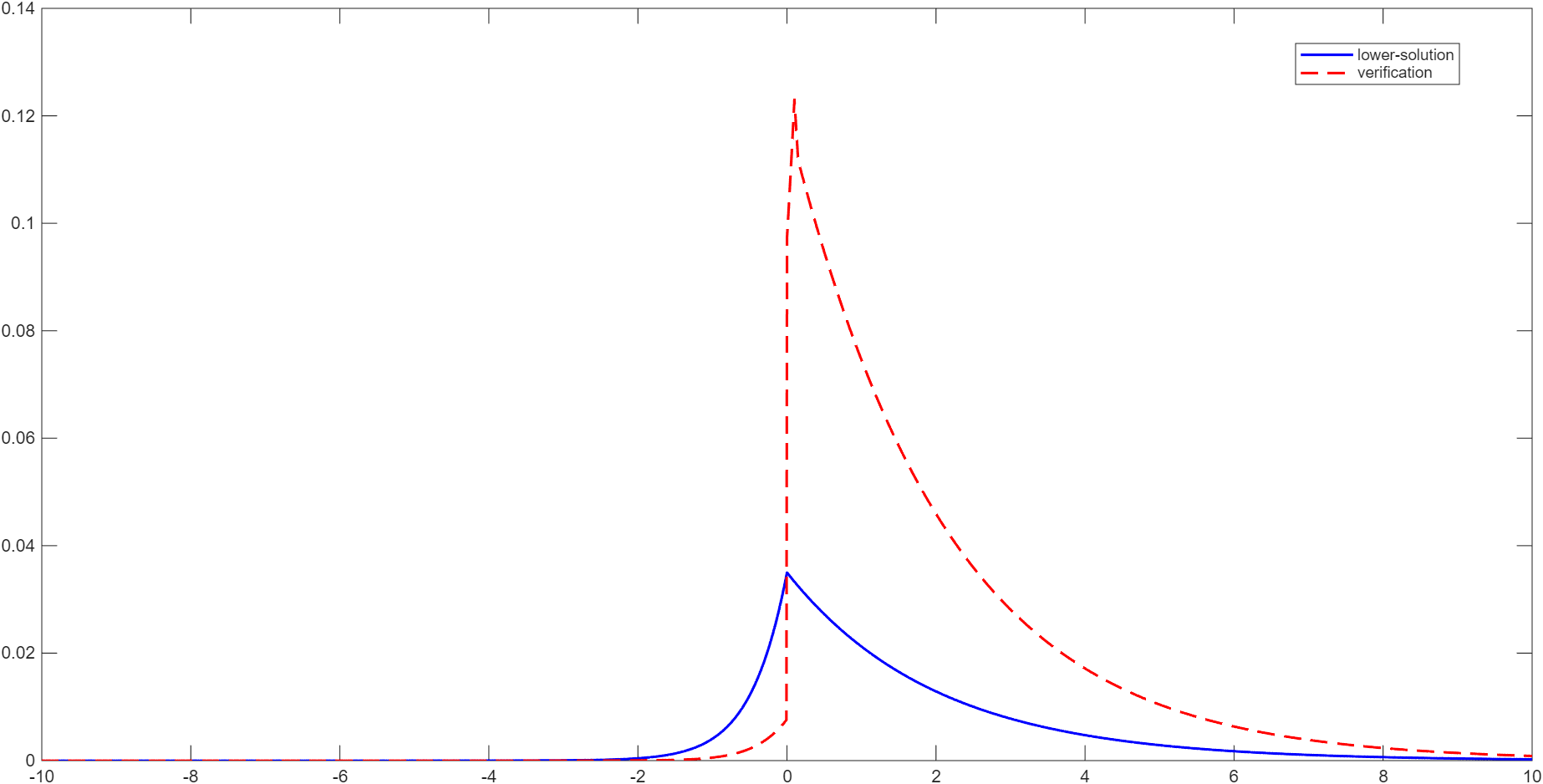}
    \caption{llustration of the lower solution
    \(\underline\varphi\) defined in
    \eqref{lower_r_le_sigma} (solid curve) and its residual
    $\mathcal R[\underline\varphi](t)
        =
        -\underline\varphi'(t)
        -\delta\underline\varphi(t)
        -H\underline\varphi(t-\sigma)
        +\rho\underline\varphi(t-r)
        e^{-\underline\varphi(t-r)}
    $(dashed curve) for $r=0.10$, $\sigma=0.15$, $\rho=6$, $\delta=1$ and $H=2$. The nonnegativity of the residual for almost every
    $t\in\mathbb R$ illustrates the lower-solution inequality.}
    \label{fig:lower_solution_verification}
\end{figure}

\subsection{Existence of the heteroclinic connection}

\begin{theorem}
\label{theorem:r_le_sigma}
Suppose that $1<\frac{\rho}{\delta+H}\leq e$. Let \(\sigma_0>0\) be the unique number satisfying $H\sigma_0e^{1+\delta\sigma_0}=1$.
If $0<r\leq\sigma\leq\sigma_0$, then equation \eqref{eq:scalar_dde} admits a monotone heteroclinic solution
\(n_*:\mathbb R\to\mathbb R\) connecting \(0\) to $\kappa$.
More precisely,
\[
n_*'(t)\geq0,
\qquad
\lim_{t\to-\infty}n_*(t)=0,
\qquad
\lim_{t\to+\infty}n_*(t)=\kappa.
\]
\end{theorem}

\begin{proof}
By Lemma \ref{lemma:H1_H2_r_le_sigma}, hypotheses \((H1)\) and \((H2)\)
hold. By Lemma
\ref{lemma:upper_r_le_sigma}, \(\overline\varphi\) is an upper solution and
belongs to \(\Gamma\). Finally, Lemma
\ref{lemma:lower_compatible_r_le_sigma} provides a lower solution compatible
with \(\overline\varphi\) in the sense of \((C1)\)--\((C3)\). All the
hypotheses of Theorem 2.1 in \cite{GomezGuayasamin2025} are therefore
satisfied, and the conclusion follows.
\end{proof}

\begin{remark}
The positive exponentially decaying right tail in
\eqref{lower_r_le_sigma} is essential when \(r<\sigma\). Indeed, for
\(r\leq t<\sigma\), it ensures that
\(\underline\varphi(t-r)>0\), thereby avoiding the obstruction produced
by the truncated lower solution used in the regime \(\sigma<r\). The
corresponding positive contribution allows the delayed harvesting term
to be controlled in the lower-solution inequality. When \(r=\sigma\),
the intermediate interval \(r\leq t<\sigma\) is empty, so the same
construction also covers the equal-delay case. Consequently, no
additional condition of the form $ H>\rho e^{2\lambda(\sigma-r)}$
is required when \(0<r\leq\sigma\).
\end{remark}

\begin{remark}
The lower solution is increases on \((-\infty,0)\) and
decreases on \((0,+\infty)\). In fact,
\[
\underline\varphi'(0^-)=A(\lambda-q\nu)>0,
\qquad
\underline\varphi'(0^+)=A(1-q)\eta<0.
\]
The relevant point, however, is not the lack of monotonicity itself, but
the behavior at the matching point. Indeed,
\[
\underline\varphi'(0^-)
>
\underline\varphi'(0^+),
\]
so the derivative jump condition
\[
\underline\varphi'(0^-)
\leq
\underline\varphi'(0^+)
\]
used in \cite{WuZouErratum2001} is not satisfied.

This does not affect the present argument because the notion of lower
solution in Theorem 2.1 of \cite{GomezGuayasamin2025} only requires the
differential inequality to hold almost everywhere and does not impose a
junction condition on the one-sided derivatives. The lower solution is
used as a positive ordered barrier preventing the monotone iteration
sequence from converging to the trivial equilibrium.
\end{remark}

\section*{Acknowledgements}
This work forms part of César Guayasamín's Ph.D. thesis in the Doctoral Program in Applied Mathematics at Universidad del Bío-Bío.

\noindent This research has been partially supported by   the internal regular project RE2448106, funded by the Universidad del B\'io-B\'io (A. Gomez), by the National Agency for Research and Development, ANID-Chile through FONDECYT Postdoctorado project 3240354 (H. Elorreaga) and by National Agency for Research and Development, ANID-Chile, Scholarship Program, Doctorado Becas Chile 2025 -- 21252578 (C. Guayasam\'in).


\end{document}